\documentclass[reqno]{amsart}
\usepackage{amsmath}
\usepackage{amsfonts,amscd, amssymb}
\usepackage{hyperref}
\usepackage{xcolor}
\usepackage{tikz-cd}
\usepackage{amsthm,latexsym,euscript,amscd}

\usetikzlibrary{matrix}
\newtheorem{theorem}{Theorem}[section]

\newtheorem{corollary}[theorem]{Corollary}
\newtheorem{proposition}[theorem]{Proposition}
\newtheorem{remark}[theorem]{Remark}
\theoremstyle{definition}
\newtheorem{definition}[theorem]{Definition}

\numberwithin{equation}{section} 
\makeatother

\DeclareMathOperator{\Idb}{{\mathbb I}}

\DeclareMathOperator{\Rdb}{{\mathbb R}}

\DeclareMathOperator{\Al}{{\mathcal A}}

\newcommand{\Sp}[1]{\operatorname{Sp}(#1)}

\newcommand{\norm}[1]{\left\Vert#1\right\Vert}

\numberwithin{equation}{section}

\begin{document}

\title[Operator entropies in JB-algebras]{Relative operator entropies and Tsallis relative operator entropies  in JB-algebras}

\author{Shuzhou Wang}
\address{Department of Mathematics, University of Georgia, Athens, GA, 30602}
\email{szwang@uga.edu}
\author{Zhenhua Wang}
\address{Department of Mathematics, University of Georgia, Athens, GA, 30602}
\email{ezrawang@uga.edu}


\date{}

\begin{abstract}
	We initiate the study of relative operator entropies and Tsallis relative operator
	entropies in the setting of JB-algebras. 
	We establish their basic properties and 
	extend the operator inequalities on relative operator
	entropies and Tsallis relative operator entropies to this setting. 
	In addition, we improve the 
	lower and upper bounds of the relative operator $(\alpha, \beta)$-entropy 
	in the setting of JB-algebras that were established in 
	Hilbert space operators setting by Nikoufar \cite{Nikoufar14, Nikoufar2020}. 
	Though we employ the same notation as in the 
	classical setting of Hilbert space operators, 
	the inequalities in the setting of JB-algebras have different 
	connotations and their proofs requires techniques in JB-algebras. 
\end{abstract}

\maketitle

\section{Introduction}
Motivated by the study of quantum mechanics, Jordan, von Neumann, and Wigner 
investigated finite dimensional Jordan algebras in \cite{Jordan34}. 
Later, von Neumann studied the infinite dimensional Jordan algebras \cite{Neumann1936algebraic}. 
In \cite{Segal1947}, Segal initiated the study of JC-algebras, and  
Effros and St{\o}rmer \cite{ES67}, St{\o}rmer \cite{stormer1965, stormer1966, stormer1968} and Topping \cite{topping1965jordan}, among others, studied these algebras more thoroughly. The theory of JB-algebras was inaugurated by Alfsen, Shultz, and St{\o}rmer \cite{Alfsen78Gelfand} and later considered by many others. As a motivation for this line of research, 
observables in a quantum system constitute a JB-algebra which is non-associative, therefore JB-algebras were considered as natural objects of study for quantum system. 
In mathematics, JB-algebras also have many powerful applications in many fields, such as analysis, geometry, operator theory, etc; more information on these can be found in 
\cite{Chu, Upmeier, Upmeier1}.  
 
Entropy, as a measure of uncertainty, is a fundmental notion in  quantum information theory. A mathematical formulation of entropy was given by Segal \cite{Segal1960}. 
In order to understand the basis of Segal's notion, operator entropy $-A\log(A)$ for positive invertible operator $A$ was considered by Nakamura and Umegaki \cite{nakamura1961}. 
Later, the relative operator entropy was used by Umegaki \cite{umegaki1962}  
to study the measures of entropy and information. The concept of relative operator entropy $S(A\vert B)$ for strictly positive operators in noncommutative information theory was first introduced by Fujii and Kamei in \cite{Fujii89relative,Fujii89Uhlmann}.
As an extension of relative operator entropy, 
generalized relative operator entropy $S_{\alpha}(A\vert B)$ was studied by Furuta in \cite{Furuta04parametric}. Meanwhile, Tsallis relative operator entropy $T_{\lambda}(A\vert B)$ was investigated by Yanagi, Kuriyama and Furuichi \cite{YANAGI2005109}, 
which has the property $\lim\limits_{\lambda \to 0}T_{\lambda}(A\vert B)=S(A\vert B)$. Furthermore, the notion of relative operator $(\alpha, \beta)$-entropy was introduced by Nikoufar \cite{Nikoufar19} as follows:
\begin{align}
\label{groe2}
S_{\alpha,\beta}(A|B)=
A^{\frac{\beta}{2}}[(A^{-\frac{\beta}{2}}BA^{-\frac{\beta}{2}})^{\alpha}
\log(A^{-\frac{\beta}{2}}BA^{-\frac{\beta}{2}})]A^{\frac{\beta}{2}}
\end{align}
for invertible positive operators $A, B$ and any real numbers $\alpha, \beta.$
This has the properties that $S_{\alpha, 1}(A|B)=S_{\alpha}(A|B)$ and $S_{0, 1}(A|B)=S(A|B).$ More recently, relative operator entropies was used as a powerful tool to study quantum coherence \cite{Guo20quantifying}, a key notion in quantum information processing. 

	In \cite{Wang2020a}, we initiated the study of relative operator entropies in 
	the settings of $C^*$-algebras, real $C^*$-algebras and JC-algebras. 
	We extended operator inequalities on relative operator entropies to these settings, and 
	improved the lower and upper bounds of the relative operator entropy 
	which are new even for relative operator entropy defined on Hilbert space. 

In another preprint \cite{Wang2020b}, 
we studied operator means in the setting of JB-algebras, 
and obtained basic operator inequalities using 
{\bf nonassociative perspective function}, which is defined as follows:
\begin{align}\label{dnapf}
P_{f\triangle h}(A,B)
=\left\{h(B)^{\frac{1}{2}}f\left(\{h(B)^{-\frac{1}{2}}Ah(B)^{-\frac{1}{2}}\}\right)h(B)^{\frac{1}{2}} \right\},	
\end{align}
where $f$ and $h$ are real continuous function on a closed interval $\Idb$ with $h>0$ and $A, B$ are two elements in a unital JB-algebra with spectra contained in $\Idb$.
Our notion is a generalization of  \cite{Ebadian11perspective} for noncommuative associative case of Hilbert space operators. For commutative case, 
perspective function was inaugurated by Effros \cite{Effros09perspective}, in which  an ingenious and simple proof of the celebrated 
Lieb's concavity theorem \cite{LIEB1973267, Lieb19731938} was given.

In the present paper, we study relative operator entropies and Tsallis relative operator entropies in the setting of JB-algebras. 
We define in section 3 relative operator entropies and 
Tsallis relative operator entropies in the setting of JB-algebras and 
investigate their properties. 
In section 4, we extend the operator inequalities on relative operator
	entropies and Tsallis relative operator entropies to this setting; 
	we also improve the lower and upper bounds of the relative operator $(\alpha, \beta)$-entropy in the setting of 
	JB-algebras	that were established in Hilbert space operators setting by Nikoufar 
	\cite{Nikoufar14, Nikoufar2020} which refined the bounds for relative operator entropy obtained earlier by Fujii and Kamei \cite{Fujii89relative,Fujii89Uhlmann}.  
   Though we employ the same notation as in the classical setting of Hilbert space operators, 
   the properties and inequalities in the setting of JB-algebras 
   we establish in this paper have different connotations and 
   their proofs requires techniques in JB-algebras.

\section{Preliminaries} 

For convenience of the reader, we give some background on JB-algebras and 
fix the notation in this section. 

\begin{definition}\label{DefJa}
A {\bf Jordan algebra} $\Al$ over real number is a vector space $\Al$ over $\Rdb$ equipped with a  bilinear product $\circ$ that satisfies the following identities:
$$a\circ b =b\circ b, \,\ \,\ (a^2\circ b)\circ a=a^2\circ (b\circ a).$$
Any associative algebra $\Al$ has an underlying Jordan algebra structure with Jordan product given by 
$$a\circ b=(ab+ba)/2.$$
Jordan suablgebra of such underlying Jordan algebras is called {\bf special}. 
\end{definition}

As the important example in physics,  
$B(H)_{sa}$, the set of bounded self adjoint operators on a Hilbert space $H$,    
is a special Jordan algebra. Note that $B(H)_{sa}$ is not an associative algebra.

\begin{definition}
A concrete {\bf JC-algebra} $\Al$ is a norm-closed Jordan subalgebra of $B(H)_{sa}$.
\end{definition}
 \begin{definition}\label{DefJB}
 A {\bf JB-algebra} is a Jordan algebra $\Al$ over $\Rdb$ with a complete norm satisfying the following conditions for $A, B\in \Al:$ 
 \begin{align*}
 	\norm{A\circ B}\leq \norm{A}\norm{B},~~\norm{A^2}=\norm{A}^2,~~\mbox{and}~~\norm{A^2}\leq \norm{A^2+B^2}.	
 \end{align*}	
 \end{definition}

A JC-algebra is a JB-algebra, but the converse is not true. 
For example, the Albert algebra is a JB-algebra but not a JC-algebra, cf. \cite[Theorem 4.6]{Alfsen03Jordan}.  

\begin{definition}
Let $\Al$ be a unital JB-algebra. 
We say $A\in \Al$ is {\bf invertible} if there exists $B\in \Al,$ which is called {\bf Jordan inverse} of $A,$ such that  
\begin{align*}
	A\circ B=I \quad \mbox{and}\quad A^2\circ B=A.	
\end{align*}
 The {\bf spectrum} of $A$ is defined by 
 \begin{align*}
 \Sp{A}:=\{\lambda\in \Rdb| A-\lambda I\,\ \text{ is not invertible in} \Al \}.	
 \end{align*}
If $\Sp{A}\subset [0,\infty),$ we say $A$ is {\bf positive}, and write $A \geq 0$. 	
\end{definition}

 \begin{definition}
Let $\Al$ be a unital JB-algebra and $A, B \in \Al$.  We define a map $U_A$ on $\Al$ by
\begin{align}\label{JI}
	U_{A}B:=\{ABA\}:= 2(A\circ B)\circ A -A^2\circ B.
\end{align}
 \end{definition}
It follows from (\ref{JI}) that $U_A$ is linear, in particular, 
\begin{align}\label{LJI}
	U_A(B-C)=\{ABA\}-\{ACA\}.
\end{align} 
Note that $ABA$ is meaningless unless $\Al$ is special, in which case $\{ABA\}=ABA.$ 
The following proposition will be used repeatedly in this paper.

\begin{proposition}\cite[Lemma 1.23-1.25]{Alfsen03Jordan} \label{3inv}
Let $\Al$ be a unital JB-algebra and $A, B$ be two elements in $\Al$.
\begin{enumerate}
	\item If $B$ is positive, then $U_A(B)=\{ABA\}\geq 0.$ 
	\item If $A, B$ are invertible, then $\{ABA\}$ is invertible with inverse $\{A^{-1}B^{-1}A^{-1}\}.$
	\item If $A$ is invertible, then $U_A$ has a bounded inverse $U_{A^{-1}}.$
\end{enumerate}
\end{proposition}

For an element $A$ in $\Al$ and a continuous function $f$ on the spectrum of $A$, 
$f(A)$ is defined by functional calculus in JB-algebras (see e.g. \cite[Proposition 1.21]{Alfsen03Jordan}).
\begin{definition}
Let $f$ is a real valued continuous function $f$ on $\Rdb.$ 
\begin{enumerate}
	\item $f$ is said to be 
	{\bf operator monotone (increasing)} on a JB-algebra $\Al$ if 
	$0\leq A\leq B$ implies $f(A)\leq f(B)$.
	\item $f$ is {\bf operator convex} if for any $\lambda \in [0, 1]$ and $A, B\geq 0,$
	$$f((1-\lambda)A+\lambda B)\leq (1-\lambda) f(A)+\lambda f(B).$$
	We say that $f$ is {\bf operator concave} if $-f$ is operator convex.
\end{enumerate}
\end{definition}

\section{Relative operator entropy and Tsallis relative operator entropy}

\begin{definition}
Let $A, B$ be two positive invertible elements in a unital JB-algebra $\Al.$   The {\bf relative operator entropy} $S(A\vert B)$ is defined by 
\begin{align}
\label{jbroe}
S(A\vert B)&:=\left\{A^{\frac{1}{2}}\log\left(\left\{A^{-\frac{1}{2}}BA^{-\frac{1}{2}}\right\}\right)A^{\frac{1}{2}}\right\}.
\end{align}
For any $\lambda\in(0,1],$ the {\bf Tsallis relative operator entropy} $T_{\lambda}(A\vert B)$  is defined by 
\begin{align}
\label{jbtoe}
T_{\lambda}(A\vert B)&: =\frac{A\#_{\lambda} B-A}{\lambda},	
\end{align}
where $A\#_{\lambda} B :=\left\{A^{\frac{1}{2}}\{A^{-\frac{1}{2}}BA^{-\frac{1}{2}}\}^{\lambda}A^{\frac{1}{2}}\right\}$ is the weighted geometric mean; for more information, see \cite{Wang2020b}.	
\end{definition}

\begin{theorem}\label{tjroei}
Let $A$ and $B$ be positive invertible. Then
\begin{align}\label{jroei}
S(A\vert B)	=\int_0^1 \dfrac{A!_t B- A}{t}dt,
\end{align}
where $A!_t B:= \left((1-\lambda)A^{-1}+\lambda B^{-1}\right)^{-1}$ is the weighted harmonic mean. 
\end{theorem}

\begin{proof}
The following identity is established in the proof of \cite[Proposition 3]{Wang2020b}
for $x>0$, 
\begin{align}\label{logi}
\log x =\int_0^1 \dfrac{(1-t+tx^{-1})^{-1}-1}{t}dt.	
\end{align}

Applying functional calculus in JB-algebras (see \cite[Proposition 1.21]{Alfsen03Jordan}) to 
(\ref{logi}) and by Proposition \ref{3inv}, 
\begin{align}\label{jlogi1}
\log\left(\left\{A^{-\frac{1}{2}}BA^{-\frac{1}{2}}\right\}\right)
&=\int_0^1 \frac{[(1-t)I+t\{A^{\frac{1}{2}}B^{-1}A^{\frac{1}{2}}\}]^{-1}-I}{t}dt.	
\end{align}
Therefore, by  Proposition \ref{3inv} again, 
\begin{align}
S(A\vert B)
&=\left\{A^{\frac{1}{2}}\int_0^1 \dfrac{\left[(1-t)I+t\{A^{\frac{1}{2}}B^{-1}A^{\frac{1}{2}}\}\right]^{-1}-I}{t}dtA^{\frac{1}{2}}\right\}  
\nonumber \\
&=\int_0^1 \dfrac{\left\{A^{-\frac{1}{2}}\left[(1-t)I+t\{A^{\frac{1}{2}}B^{-1}A^{\frac{1}{2}}\}\right]A^{-\frac{1}{2}}\right\}^{-1}-A}{t}dt 
 \nonumber \\
&=\int_0^1 \dfrac{\left[(1-t)A^{-1}+t B^{-1}\right]^{-1}-A}{t}dt 
 \nonumber \\
&=\int_0^1 \dfrac{A!_t B- A}{t}dt.
\end{align}
\end{proof}

\begin{proposition}\label{jroexlogx}
Let $A$ and $B$ be positive invertible. Then
\begin{align}\label{jroexlogx1}
S(A\vert B)=\left\{B^{\frac{1}{2}}\left[-\{B^{-\frac{1}{2}}AB^{-\frac{1}{2}}\}\circ \log(\{B^{-\frac{1}{2}}AB^{-\frac{1}{2}}\})\right]B^{\frac{1}{2}}\right\}.	
\end{align}	
\end{proposition}

\begin{proof}
According to the proof of \cite[Proposition 3]{Wang2020b}, we have
\begin{align}
-x\log x 
&=\int_0^1\dfrac{[(1-t)x^{-1}+t]^{-1}-x}{t}dt. 
\label{xlogxi}
\end{align}	
Denote
\begin{align*}
E=-\left\{B^{-\frac{1}{2}}AB^{-\frac{1}{2}}\right\}\circ 
\log \left( \{B^{-\frac{1}{2}}AB^{-\frac{1}{2}}\} \right).	
\end{align*}
Utilizing functional calculus in JB-algebras for (\ref{xlogxi}),  
\begin{align}
E=\int_0^1\dfrac{[(1-t)\{B^{-\frac{1}{2}}AB^{-\frac{1}{2}}\}^{-1}+t]^{-1}-\{B^{-\frac{1}{2}}AB^{-\frac{1}{2}}\}}{t}dt
\end{align}
Therefore,
\begin{align*}
\left\{ B^{\frac{1}{2}}EB^{\frac{1}{2}} \right\}
&=\left\{B^{\frac{1}{2}}\int_0^1\dfrac{[(1-t)\{B^{-\frac{1}{2}}AB^{-\frac{1}{2}}\}^{-1}+t]^{-1}-\{B^{-\frac{1}{2}}AB^{-\frac{1}{2}}\}}{t}dtB^{\frac{1}{2}}\right\}\\
&=\int_0^1\dfrac{\left\{B^{-\frac{1}{2}}[(1-t)\{B^{-\frac{1}{2}}AB^{-\frac{1}{2}}\}^{-1}+t]B^{-\frac{1}{2}}\right\}^{-1}-A}{t}dt	\\
&=\int_0^1 \dfrac{\left[(1-t)A^{-1}+t B^{-1}\right]^{-1}-A}{t}dt\\
&=\int_0^1 \dfrac{A!_t B- A}{t}dt\\
&=S(A\vert B).
\end{align*}
\end{proof}

\begin{proposition}\label{xlogx}
Let $\Al$ be a unital JB-algebra. The functions $x\to -x\log x$ is operator concave on $(0,\infty).$
\end{proposition}

\begin{proof}
Denote $h_{\alpha}(x)=1-\alpha(\alpha +x)^{-1}- x(\alpha +1)^{-1}.$
By (4.5) in \cite{Wang2020b}, $-\alpha(\alpha +x)^{-1}$ is operator concave on $(0,\infty).$ Thus, $h_{\alpha}(x)$ is also operator concave on $(0,\infty).$
By (4.9) in \cite{Wang2020b},
\begin{align*}	
-x\log x &=\int_0^{\infty}x[(\alpha +x)^{-1}-(\alpha+1)^{-1}]d\alpha \\
&=\int_0^{\infty} h_{\alpha}(x) d\alpha
\end{align*}	
By \cite[Lemma 1]{Wang2020b}, $-x\log x$ is operator concave.
\end{proof}

\begin{proposition}
	\label{roeprop}
The relative operator entropy $S(A\vert B)$ defined in $\Al$ has the following properties:
\begin{itemize}
\item[(i)] $S(\alpha A\vert \alpha B)=\alpha S(A\vert B)$ for any positive number $\alpha.$
\item[(ii)] If $B\leq C,$	 then $S(A\vert B)\leq S(A\vert C).$
\item[(iii)] $S(A\vert B)$ is operator concave with respect to $A, B$ individually. 
\item[(iv)] $S(\{CAC\}\vert \{CBC\})=\{C S(A\vert B)C\},$ for any invertible $C$ in $\Al.$
\end{itemize}
\end{proposition}

\begin{proof}
For (i), it follows directly from the definition.

Proof of (ii). If $B\leq C,$ then
\begin{align*}
\left\{ A^{-\frac{1}{2}}BA^{-\frac{1}{2}} \right\}
\leq \left\{A^{-\frac{1}{2}}CA^{-\frac{1}{2}} \right\}.	
\end{align*}
By \cite[Proposition 5]{Wang2020b},
\begin{align*}
\log\left(\{A^{-\frac{1}{2}}BA^{-\frac{1}{2}}\}\right)\leq \log\left(\{A^{-\frac{1}{2}}CA^{-\frac{1}{2}}\}\right).	
\end{align*}
 This implies that $S(A\vert B)\leq S(A\vert C).$
 
 (iii) For any $0\leq t\leq 1,$ we denote 
 $$E=S(A\vert (1-t)B_1+tB_2).$$
 Since $\log x$ is operator concave then 
\begin{align*}
 E
 &=\left\{A^{\frac{1}{2}}\log\left((1-t)\{A^{-\frac{1}{2}}B_1A^{-\frac{1}{2}}\}+t\{A^{-\frac{1}{2}}B_2A^{-\frac{1}{2}}\}\right)A^{\frac{1}{2}}\right\}\\
 &\geq (1-t)\left\{A^{\frac{1}{2}}\log\left(\{A^{-\frac{1}{2}}B_1A^{-\frac{1}{2}}\}\right)A^{\frac{1}{2}}\right\}+t\left\{A^{\frac{1}{2}}\log\left(\{A^{-\frac{1}{2}}B_2A^{-\frac{1}{2}}\}\right)A^{\frac{1}{2}}\right\}\\
 &=(1-t)S(A\vert B_1)+tS(A\vert B_2)
 \end{align*}
 
On the other hand, we denote
\begin{align*}
F&=\left[-\{B^{-\frac{1}{2}}((1-t)A_1+tA_2)B^{-\frac{1}{2}}\}\circ \log(\{B^{-\frac{1}{2}}((1-t)A_1+tA_2)B^{-\frac{1}{2}}\})\right]\\
G&=S((1-t)A_1+tA_2\vert B)	
\end{align*}
By Proposition \ref{xlogx}, 
\begin{align*}
F&\geq (1-t)
\left [-\{B^{-\frac{1}{2}}A_1B^{-\frac{1}{2}}\}\circ \log(\{B^{-\frac{1}{2}}A_1B^{-\frac{1}{2}}\})\right]\\
&\quad+t\left[-\{B^{-\frac{1}{2}}A_2B^{-\frac{1}{2}}\}\circ \log(\{B^{-\frac{1}{2}}A_2B^{-\frac{1}{2}}\})\right]
\end{align*}
From Propsotion \ref{jroexlogx}, one sees that
\begin{align*}
G&=\left\{ B^{\frac{1}{2}} FB^{\frac{1}{2}} \right\}\\
&\geq (1-t)\left\{B^{\frac{1}{2}}\left [-\{B^{-\frac{1}{2}}A_1B^{-\frac{1}{2}}\}\circ \log(\{B^{-\frac{1}{2}}A_1B^{-\frac{1}{2}}\})\right]B^{\frac{1}{2}}\right\}\\
&\quad +t\left\{B^{\frac{1}{2}}\left [-\{B^{-\frac{1}{2}}A_2B^{-\frac{1}{2}}\}\circ \log(\{B^{-\frac{1}{2}}A_2B^{-\frac{1}{2}}\})\right]B^{\frac{1}{2}}\right\}\\
&= (1-t)S(A_1\vert B)+tS(A_2\vert B)	
\end{align*}

Proof of (iv). 
By definition of harmonic mean in \cite{Wang2020b} and Proposition \ref{3inv}, 
for any invertible element $C$ in $\Al$
\begin{align}
\{CAC\}!_t \{CABC\}&=\left[(1-t)\{CAC\}^{-1}+t\{CBC\}^{-1}\right]^{-1}\nonumber\\
&=\left\{C^{-1}[(1-t)A^{-1}+tB^{-1}]C^{-1}\right\}^{-1}\nonumber	\\
&=\left\{C[(1-t)A^{-1}+tB^{-1}]^{-1}C\right\}\nonumber\\
&=\{C (A!_tB) C\}\label{amm}
\end{align}
According to Theorem \ref{tjroei} and (\ref{amm})
\begin{align*}
S(\{CAC\}\vert \{CBC\})
&=\int_0^1 \dfrac{\{C (A!_tB) C\}-\{CAC\}}{t}dt\\
&=\left\{C\int_0^1 \dfrac{\{(A!_tB-A) \}}{t}dt~C\right\}\\
&=\{C S(A\vert B) C\}.	
\end{align*}
\end{proof}

\begin{proposition}
	\label{joconcav}
	Let $\Al$ be   a JC-algebra. The relative operator entropy 
	$S(A \vert B)$ is jointly operator concave. 	
\end{proposition}

\begin{proof}
A JC-algebra can be realized as self-adjoint operators on a Hilbert space. 
Since $f(x)=\log x$ and $h(x)=x$ are operator concave, then by \cite[Corollary 2.6]{Ebadian11perspective}, the relative operator entropy 
$S(A \vert B)=P_{f\triangle h}(B,A)$ is jointly operator concave for operators $A$ and $B$ on a Hilbert space. 	   
\end{proof}

\begin{theorem}\label{Ptroeint}
Let $A, B$ be two positive invertible elements in a unital JB-algebra $\Al.$ For any $\lambda \in (0, 1),$
\begin{align}\label{troeint}
T_{\lambda}(A\vert B)=\dfrac{\sin(\lambda \pi)}{\lambda\pi}\int_0^1 \left(\dfrac{t}{1-t}\right)^{\lambda}\frac{A!_t B -A}{t}dt.	
\end{align}	
\end{theorem}
\begin{proof}
By 	\cite[Theorem 2]{Wang2020b}, 
\begin{align}\label{geomint}
A\#_{\lambda} B=\dfrac{\sin(\lambda \pi)}{\pi} \int_0^1 \dfrac{t^{\lambda -1}}{(1-t)^{\lambda}}(A!_t B)dt.	
\end{align}
Applying functional calculus in JB-algebras to the following identity of $\Gamma$-function,  
\begin{align}\label{indint}
\dfrac{\sin(\lambda \pi)}{\pi} \int_0^1 \dfrac{t^{\lambda -1}}{(1-t)^{\lambda}}dt=1,
 \end{align}
 we have
 \begin{align}\label{JBind}
\dfrac{\sin(\lambda \pi)}{\pi} \int_0^1 \dfrac{t^{\lambda -1}}{(1-t)^{\lambda}} Adt=A.	
 \end{align}

 Combining (\ref{geomint}) and (\ref{JBind}),
 \begin{align}
 \frac{A\#_{\lambda} B-A	}{\lambda}=\dfrac{\sin(\lambda \pi)}{\lambda\pi}\int_0^1 \left(\dfrac{t}{1-t}\right)^{\lambda}\frac{A!_t B -A}{t}dt.
 \end{align}
\end{proof}
\begin{proposition}
The Tsallis relative operator entropy $ T_{\lambda}(A\vert B)$ defined in $\Al$ has the following properties: 
\begin{enumerate}
\item[(i)] $T_{\lambda}(\alpha A\vert \alpha B)=\alpha T_{\lambda}(A\vert B)$ for any positive number $\alpha$
\item[(ii)] If $B\leq D,$ then $ T_{\lambda}(A\vert B)\leq  T_{\lambda}(A\vert D).$ 
\item[(iii)] $T_{\lambda}(\{C AC\}\vert \{C BC\})=\{C T_{\lambda}(A\vert B) C\}$ for any invertible $C$ in $\Al.$ 
\item[(iv)]  $T_{\lambda}(A\vert B)$ is operator concave with respect to $A,$ $B$ individually. 
\item[(v)] $\lim\limits_{\lambda \to 0}T_{\lambda}(A\vert B)=S(A\vert B).$	
\end{enumerate}
\end{proposition}
\begin{proof}
By \cite[Proposition 6(i)]{Wang2020b}, $(\alpha A)\#_{\lambda} (\alpha B)=\alpha (A\#_{\lambda} B).$ 
Then (i) follows immediately.

For (ii), it follows from \cite[Proposition 6(ii)]{Wang2020b} that 
\begin{align*}
A\#_{\lambda} B-A\leq A\#_{\lambda} D-A.
\end{align*}
Then, $ T_{\lambda}(A\vert B)\leq  T_{\lambda}(A\vert D).$

(iii) According to \cite[Proposition 6(iv)]{Wang2020b},
 \begin{align*}
T_{\lambda}(\{C AC\}\vert \{C BC\})&=\dfrac{\{C AC\}\#_{\lambda}\{C BC\}-\{C AC\}}{\lambda}	\\
&=\dfrac{\{C (A\#_{\lambda} B) C\}-\{C AC\}}{\lambda}\\
&=\{C T_{\lambda}(A\vert B) C\}.
\end{align*}

Proof of (iv), it follows from the fact $A\#_{\lambda} B$ is operator concave with respect to $A,$ $B$ individually (See e.g. \cite[Proposition 6(iii)]{Wang2020b}).

(v) Denote $\ln_{\lambda}x= \frac{x^{\lambda}-1}{\lambda}.$ By Dini's theorem, $\ln_{\lambda}x$ uniformly converges to $\log x$ on any bounded closed interval $[a, b]\subset [0,\infty).$ 
It implies that 
\begin{align*}
\lim\limits_{\lambda \to 0}\ln_{\lambda}(\{A^{-\frac{1}{2}}BA^{-\frac{1}{2}}\})=\lim\limits_{\lambda \to 0}\frac{\{A^{-\frac{1}{2}}BA^{-\frac{1}{2}}\}^{\lambda}-1}{\lambda }=\log\left(\left\{A^{-\frac{1}{2}}BA^{-\frac{1}{2}}\right\}\right)	
\end{align*}
Since $U_{A^{\frac{1}{2}}}$ is continuous, then
\begin{align*}
\lim\limits_{\lambda \to 0}T_{\lambda}(A\vert B)
&=\left\{A^{\frac{1}{2}}\log\left(\left\{A^{-\frac{1}{2}}BA^{-\frac{1}{2}}\right\}\right)A^{\frac{1}{2}}\right\}=S(A\vert B).
\end{align*}
\end{proof}
Similar argument as in Proposition \ref{joconcav} gives
\begin{proposition}
	\label{joconcav1}
	Let $\Al$ be   a JC-algebra. The Tsallis relative operator entropy 
	$T_{\lambda}(A \vert B)$ is jointly operator concave. 	
\end{proposition}

\section{Upper and lower bounds of generalized relative operator entropies and Tsallis relative operator entropies}

\begin{definition}
Let $A, B$ be positive invertible elements in a unital JB-algebra $\Al$.
The {\bf relative operator $(\alpha, \beta)$-entropy} $S_{\alpha,\beta}(A|B)$ and {\bf Tsallis relative operator $(\lambda, \beta)$-entropy} $T_{\lambda,\beta}(A\vert B)$ 
are defined respectively by 
\begin{align}\label{groe2}
S_{\alpha,\beta}(A|B)&:=
\left\{A^{\frac{\beta}{2}}\left[\{A^{-\frac{\beta}{2}}BA^{-\frac{\beta}{2}}\}^{\alpha}
\circ \log\left(\{A^{-\frac{\beta}{2}}BA^{-\frac{\beta}{2}}\}\right)\right]
A^{\frac{\beta}{2}}\right\} ,	\\
T_{\lambda,\beta}(A\vert B)&:=\left\{A^{\frac{\beta}{2}}
\ln_{\lambda}\left(\left\{A^{-\frac{\beta}{2}}BA^{-\frac{\beta}{2}}\right\}\right)
A^{\frac{\beta}{2}}\right\}	.
\end{align}	
\end{definition}

Clearly, $S_{0, 1}(A|B)=S(A|B)$ and $T_{\lambda,1}(A\vert B)=T_{\lambda}(A\vert B).$  

\medskip

In this section, we study the bounds of  
$S_{\alpha,\beta}(A|B)$ and $T_{\lambda,\beta}(A\vert B)$ 
in the setting of JB-algebras. For real numbers $\alpha\geq 0$ and $\beta>0,$ we set
 
\begin{align}
\label{I}
{\rm I}&=2\left\{A^{\frac{\beta}{2}}
\left[\left(1-2(1+\{A^{-\frac{\beta}{2}}B A^{-\frac{\beta}{2}}\})^{-1}\right)\circ\{A^{-\frac{\beta}{2}}B A^{-\frac{\beta}{2}}\}^{\alpha}\right]A^{\frac{\beta}{2}}\right\}\\
\label{II}
{\rm II}&=4A\#_{(\alpha, \beta)}B-8\left\{A^{\frac{\beta}{2}}\left[\{A^{-\frac{\beta}{2}}B A^{-\frac{\beta}{2}}\}^{\alpha}\circ ( \{A^{-\frac{\beta}{2}}B A^{-\frac{\beta}{2}}\}^{\frac{1}{2}}+1   )^{-1}\right]A^{\frac{\beta}{2}}\right\}\\
\label{III}
{\rm III}&=A\#_{(\alpha+\frac{1}{2}, \beta)}B-A\#_{(\alpha-\frac{1}{2}, \beta)}B\\
\label{V}
{\rm V}&=\frac{1}{2}(A\#_{(\alpha+1, \beta)}B-A\#_{(\alpha-1, \beta)}B) 		
\end{align}
where $A\#_{(\alpha, \beta)}B: =\left\{A^{\frac{\beta}{2}}\left(\{A^{-\frac{\beta}{2}}BA^{-\frac{\beta}{2}}\}\right)^{\alpha }A^{\frac{\beta}{2}}\right\}$ is the {\bf operator $(\alpha, \beta)$-geometric mean}. Then ${\rm I}, {\rm II}, {\rm III}$, and ${\rm V}$ defined above are in $\Al$.

 For convenience of the reader, we recall  here Theorem 1 from  \cite{Wang2020b}:
 
\begin{theorem}
	\label{tnapf} 
	Let $\Al$ be a unital JB-algebra.	
	Let $r, q$ and $h$ be real valued continuous functions on a closed interval $\Idb$ such that $h>0$ 
	and $r(x)\leq q(x)$. For elements $A$ 
	and $B$ in $\Al$ such that the spectra of $B$ and $\{h(B)^{-1/2}Ah(B)^{-1/2}\}$ are contained in $\Idb$, 
	\begin{equation}\label{inapf} 
		P_{r\triangle h}(A,B)\leq P_{q\triangle h}(A, B).
	\end{equation}
\end{theorem}

We extend \cite[Proposition 2.5]{Wang2020a} to the setting of JB-algebras. 
\begin{proposition}\label{orlurnp} 
Let $A$ and $B$ be two positive invertible elements in a unital JB-algebras $\Al$. 
Then for $\alpha\geq 0$ and $\beta>0,$ 
\begin{enumerate}
\item[(i)] $A\#_{(\alpha, \beta)}B-A\#_{(\alpha-1, \beta)}B\leq 
{\rm I}\leq A\#_{(\alpha+1, \beta)}B-A\#_{(\alpha, \beta)}B.$

\item[(ii)] $A\#_{(\alpha, \beta)}B-A\#_{(\alpha-1, \beta)}B\leq 
{\rm V}\leq A\#_{(\alpha+1, \beta)}B-A\#_{(\alpha, \beta)}B.$	

\item[(iii)] $A\#_{(0, \beta)}B-A\#_{(-1, \beta)}B\leq A\#_{(\alpha, \beta)}B-A\#_{(\alpha-1, \beta)}B.$
\end{enumerate}
\end{proposition}
\begin{proof}
Proof of (i).  
Let 
\begin{align*}
r(x)&= x^{\alpha}-x^{\alpha-1},\\
q(x)&=2\left( 1-\frac{2}{x+1} \right)	x^{\alpha},\\
k(x)&=x^{\alpha+1}-x^{\alpha}.
\end{align*}
Then, 
\begin{equation}
\label{rqk1}
r(x)\leq q(x)\leq k(x)
\end{equation}
for $x>0$ as shown in the proof of  \cite[Proposition 2.5]{Wang2020a}.

Denoting $h(x)=x^{\beta}$,  we have 
\begin{align*}
P_{r\triangle h}(B, A)&=A\#_{(\alpha, \beta)}B-A\#_{(\alpha-1, \beta)}B \\
P_{q\triangle h}(B, A)&
=2\left\{A^{\frac{\beta}{2}}
\left[\left(1-2(1+\{A^{-\frac{\beta}{2}}B A^{-\frac{\beta}{2}}\})^{-1}\right)\circ\{A^{-\frac{\beta}{2}}B A^{-\frac{\beta}{2}}\}^{\alpha}\right]A^{\frac{\beta}{2}}\right\}
={\rm I}\\
P_{k\triangle h}(B, A)&=A\#_{(\alpha+1, \beta)}B-A\#_{(\alpha, \beta)}B.
\end{align*}

Applying Theorem \ref{tnapf} to the inequalities (\ref{rqk1}), we obtain desired inequalities. 

For (ii), the inequalities 
\begin{align} \label{rqk2}
 x^{\alpha}-x^{\alpha-1}\leq \frac{1}{2}(x^{\alpha+1}-x^{\alpha-1})\leq  x^{\alpha+1}-x^{\alpha}
\end{align}
hold for all $x>0.$ 
Using the perspective functions associated with the three functions in (\ref{rqk2}) with $h$ as in 
the proof of (i) and applying Theorem \ref{tnapf}, (ii) follows.

(iii) Note that $1-\frac{1}{x}\leq x^{\alpha}-x^{\alpha-1}$ for all $x>0.$ Applying Theorem \ref{tnapf} to this inequality with $h(x)=x^{\beta},$ we know (iii) is true. 
\end{proof}

Let 
\begin{align}
r_{\delta}(x)&= \left[ \ln \delta+2 \left( 1-\dfrac{2\delta}{x+\delta} \right) \right] x^{\alpha} \nonumber\\
s_{\delta}(x)&=\left[ \ln \delta+4-\dfrac{8\sqrt{\delta}}{\sqrt{x}+\sqrt{\delta}} \right] x^{\alpha}\nonumber\\
q(x)&=x^{\alpha}\ln(x) \label{rrsqjk}\\
j_{\delta}(x)&=x^{\alpha+\frac{1}{2}}\dfrac{1}{\sqrt{\delta}}-x^{\alpha-\frac{1}{2}}\sqrt{\delta}+x^{\alpha}\ln\delta \nonumber\\
k_{\delta}(x)&=\dfrac{x^{\alpha+1}}{2\delta}-\dfrac{x^{\alpha -1}}{2}\delta+x^{\alpha}\ln \delta \nonumber
\end{align}

Ultilizing refined Hermite-Hadamard inequality (see e.g. \cite[Lemma 2.7]{Wang2020a}), we obtained the following proposition, which is crucial for future results.

\begin{proposition}\label{PrHHi}
Let $r_{\delta}(x),
 s_{\delta}(x),
 q(x),
 j_{\delta}(x),
 k_{\delta}(x)$ be defined as above.
\begin{itemize}
\item[(a)] If $\alpha\geq 0$ and $x\geq \delta\geq 1,$ then
\begin{align*}
r_{\delta}(x)
\leq s_{\delta}(x)
\leq q(x)
\leq j_{\delta}(x)
\leq k_{\delta}(x).
\end{align*}
\item[(b)] If $\alpha\geq 0$ and $x\geq \delta\geq 1,$ then
\begin{align*}
s_1(x)\leq  s_{\delta}(x)\,\ \mbox{and} \,\ j_{\delta}(x)\leq j_1(x).	
\end{align*}
\item[(c)] If $\alpha\geq 0$ and $x\leq \delta\leq 1,$ then
\begin{align*}
k_{\delta}(x)
\leq j_{\delta}(x)
\leq q(x)
\leq s_{\delta}(x)
\leq r_{\delta}(x)
\end{align*}	
\item[(d)] If $\alpha\geq 0$ and $x\leq \delta\leq 1,$ then
\begin{align*}
s_{\delta}(x)\leq  s_1(x) \,\  \mbox{and} \,\ j_1(x)\leq j_{\delta}(x).	
\end{align*}		
\end{itemize}

\end{proposition}
\begin{proof}
One could find them in the proof of Theorem 2.8, Theorem 2.11, Theorem 3.1 and Proposition 3.2 in \cite{Wang2020a}.	
\end{proof}


 \begin{theorem}\label{ulbroe}
 Let $A$ and $B$ be two positive invertible elements in a unital JB-algebra $\Al,$ $\alpha\geq 0$ and $\beta>0.$
 \begin{itemize}
 \item[(i)] If $A^{\beta}\leq B,$ then
 \begin{align}
\label{main1}
{\rm I}
\leq {\rm II}
\leq  S_{\alpha, \beta}(A|B)
\leq {\rm III}
\leq {\rm V}
\end{align}
  \item[(ii)] If $A^{\beta}\geq B,$ 	then
  \begin{align}
\label{main2}
{\rm V}
\leq {\rm III}
\leq  S_{\alpha, \beta}(A|B)
\leq {\rm III}
\leq {\rm I}
\end{align}	
 \end{itemize}
 \end{theorem}
\begin{proof}
(i). Applying Theorem \ref{tnapf} to Proposition \ref{PrHHi} (a) with $\delta=1$ and with $h(t)=t^{\beta}$, we derive the inequality
\begin{align*}
 {\rm I}\leq {\rm II}\leq  S_{\alpha, \beta}(A|B)\leq {\rm III}\leq {\rm V}.
\end{align*} 

Proof of (ii). Using the perspective functions associated with the five functions in Proposition \ref{PrHHi} (c) with $\delta=1$ and applying Theorem \ref{tnapf}, (\ref{main2}) follows.
\end{proof}
Theorem \ref{ulbroe} above improves the upper and lower bounds of relative operator $(\alpha, \beta)$-entropy established by Nikoufar \cite{Nikoufar2020} and extends it to the setting of JB-algebras. 

Combining Theorem \ref{ulbroe} and Proposition \ref{orlurnp}, we have 
\begin{corollary}\label{ulbroe4}
 Let $A$ and $B$ be two positive invertible elements in a unital JB-algebra $\Al,$ $\alpha\geq 0$ and $\beta>0.$
 \begin{itemize}
 \item[(i)] If $A^{\beta}\leq B,$ then
 \begin{align*}
A\#_{(\alpha, \beta)}B-A\#_{(\alpha-1, \beta)}B
&\leq {\rm I}
\leq {\rm II}\\
&\leq  S_{\alpha, \beta}(A|B)\\
&\leq {\rm III}
\leq {\rm V}\\
&\leq A\#_{(\alpha+1, \beta)}B-A\#_{(\alpha, \beta)}B
\end{align*}
\item[(ii)] If $A^{\beta}\geq B,$ 	then
  \begin{align*}
A\#_{(\alpha, \beta)}B-A\#_{(\alpha-1, \beta)}B
&\leq  {\rm V}\leq {\rm III}\\
&\leq S_{\alpha, \beta}(A|B)\\
&\leq {\rm II}\leq {\rm I}\\
&\leq A\#_{(\alpha +1 , \beta)}B-A\#_{(\alpha, \beta)}B.
\end{align*}
\end{itemize}
 \end{corollary}

The following result refines the upper and lower bounds of relative operator entropy obtained by Nikoufar \cite{Nikoufar14, Nikoufar2020}, which improved the bounds established by Fujii and Kamei \cite{Fujii89relative, Fujii89Uhlmann}, and extends it to the setting of JB-algebras. 
\begin{corollary}\label{ulbroe1}
 Let $A$ and $B$ be two positive invertible elements in a unital JB-algebra $\Al,$ $\alpha\geq 0$ and $\beta>0.$
 \begin{itemize}
 \item[(i)] If $A\leq B,$ then
\begin{align*}
A-\left\{AB^{-1}A\right\}
&\leq 2\left(A-2\{A(A+B)^{-1}A\}\right)\\
&\leq 4A-8\left\{A(A\#_{\frac{1}{2}}B+A)^{-1}A\right\}\\
&\leq S(A|B)\\
&\leq A\#_{\frac{1}{2}}B-A\#_{-\frac{1}{2}}B\\
&\leq \frac{1}{2}\left(B-\{AB^{-1}A\}\right)\\
&\leq B-A.
\end{align*}
\item[(ii)] If $A\geq B,$ 	then
\begin{align*}
A-\{AB^{-1}A\}&\leq \frac{1}{2}\left(B-\{AB^{-1}A\}\right)\\
&\leq A\#_{\frac{1}{2}}B-A\#_{-\frac{1}{2}}B\\
&\leq S(A|B)\\
&\leq 4A-8\{A(A\#_{\frac{1}{2}}B+A)^{-1} A\}\\
&\leq 2\left(A-2\{A(A+B)^{-1}A\}\right)\\
&\leq B-A.
\end{align*}
\end{itemize}
\end{corollary}

Suppose that $A, B$ are two positive invertible elements. For any real number $\delta>0,$ $\alpha\geq 0$ and $\beta>0,$ we denote
\begin{align*}
{\rm I'}&=(\ln\delta+2)A\#_{(\alpha,\beta)} B
-2\delta \left\{A^{\frac{\beta}{2}}[(\{A^{-\frac{\beta}{2}}BA^{-\frac{\beta}{2}}\}+\delta)^{-1}\circ\{A^{-\frac{\beta}{2}}BA^{-\frac{\beta}{2}}\}^{\alpha}]A^{\frac{\beta}{2}}\right\}\\
{\rm II'}&=(\ln\delta+4)A\#_{(\alpha,\beta)} B
-8\sqrt{\delta}\left\{A^{\frac{\beta}{2}}[(\{A^{-\frac{\beta}{2}}BA^{-\frac{\beta}{2}}\}^{\frac{1}{2}}+\sqrt{\delta})^{-1}\circ \{A^{-\frac{\beta}{2}}BA^{-\frac{\beta}{2}}\}^{\alpha}]A^{\frac{\beta}{2}}\right\}\\
{\rm III'}&=(\frac{1}{\sqrt{\delta}}A\#_{(\alpha+\frac{1}{2},\beta)}B-\sqrt{\delta} A\#_{(\alpha-\frac{1}{2},\beta)}B)+\ln\delta A\#_{(\alpha,\beta)}B\\
{\rm V'}&=\frac{1}{2}(\frac{1}{\delta}A\#_{(\alpha+1,\beta)}B-\delta A\#_{(\alpha-1,\beta)}B)+\ln\delta A\#_{(\alpha,\beta)}B.
\end{align*}
\begin{theorem}\label{ulbroe3}
 Let $A$ and $B$ be two positive invertible elements in a unital JB-algebra $\Al,$ $\alpha\geq 0$ and $\beta>0.$
 \begin{itemize}
\item[(i)] If $\delta \geq 1$ and $\delta A^{\beta}\leq B,$ then 
\begin{align*}
{\rm I'}\leq {\rm II'}\leq S_{\alpha, \beta}(A|B)\leq {\rm III'}\leq {\rm V'}.	
\end{align*}
\item[(ii)] If $\delta \geq 1$ and $\delta A^{\beta}\leq B,$ then 
\begin{align*}
{\rm II}\leq {\rm II'}\quad \mbox{and}\quad {\rm III'}\leq {\rm III}.
\end{align*}	
\item[(iii)] If $\delta\leq 1$ and $\delta A^{\beta}\geq B,$ then 
\begin{align*}
{\rm V'}\leq {\rm III'}\leq  S_{\alpha, \beta}(A|B)\leq {\rm II'}\leq {\rm I'}.
\end{align*}
\item[(iv)] If $\delta\leq 1$ and $\delta A^{\beta}\geq B,$ then 
\begin{align*}
{\rm II'}\leq {\rm II}\quad \mbox{and}\quad {\rm III}\leq {\rm III'}.
\end{align*}		
\end{itemize}
\end{theorem}
\begin{proof}
(i). Using perspective functions associated with the five functions in Propositon \ref{PrHHi}(a) with $h(t)=t^{\beta},$ and applying Theorem \ref{tnapf}, (i) follows. 	

Proof of (ii). Applying Theorem \ref{tnapf} to Propositon \ref{PrHHi}(b) with $h$ as in the proof of (i), we obtain the inequalities 
\begin{align*}
{\rm II}\leq {\rm II'}\quad \mbox{and}\quad {\rm III'}\leq {\rm III}.
\end{align*}	

(iii) follows by applying Theorem \ref{tnapf} to Propositon \ref{PrHHi}(c). Similar argument as in the proof of (ii) gives (iv).
\end{proof} 

The following corollary improves Corollary \ref{ulbroe1} and also sharply refines the lower and upper bounds of the relative operator entropy established by Nikoufar in \cite{Nikoufar14, Nikoufar2020},  
which refined the bounds obtained 
earlier by Fujii and Kamei \cite{Fujii89relative,Fujii89Uhlmann}.
\begin{corollary}
Let $A$ and $B$ be two positive invertible elements in a unital JB-algebra $\Al$.
\begin{itemize}
\item[(i)] If $\delta\geq 1$ and $\delta A\leq B$, then 
\begin{align*}
A-\{AB^{-1}A\}
&\leq 2\left[A-2\{A(A+B)^{-1}A\}\right]\\
&\leq (\ln\delta)A+2\left[A-2\delta \{A(\delta A+B)^{-1}A\}\right]\\
&\leq (\ln\delta+4)A-8\sqrt{\delta}\left\{A(A\#_{\frac{1}{2}}B+\sqrt{\delta}A)^{-1}A\right\} \\
&\leq S(A|B)\\
&\leq \left( \frac{1}{\sqrt{\delta}}A\#_{\frac{1}{2}}B-\sqrt{\delta} A\#_{-\frac{1}{2}}B \right)+(\ln\delta) A\\
&\leq \frac{1}{2} \left( \frac{1}{\delta}B-\delta \{AB^{-1}A\} \right)+(\ln\delta) A\\
&\leq \frac{1}{2}\left(B-\{AB^{-1}A\}\right)\\
&\leq B-A.
\end{align*}
\item[(ii)] If $\delta\leq 1$ and $\delta A\geq B$, then 
\begin{align*}
A-\{AB^{-1}A\}
&\leq \frac{1}{2}(B-\{AB^{-1}A\})\\
&\leq \frac{1}{2} \left( \frac{1}{\delta}B-\delta \{AB^{-1}A\} \right)+(\ln\delta) A\\
&\leq \left( \frac{1}{\sqrt{\delta}}A\#_{\frac{1}{2}}B-\sqrt{\delta} A\#_{-\frac{1}{2}}B \right)+(\ln\delta) A\\
&\leq  S(A|B)\\
&\leq (\ln\delta+4)A-8\sqrt{\delta}\left\{A(A\#_{\frac{1}{2}}B+\sqrt{\delta}A)^{-1}A\right\} \\
&\leq (\ln\delta)A+2\left[A-2\delta \{A(\delta A+B)^{-1}A\}\right]\\
&\leq 2\left[A-2\{A(A+B)^{-1}A\}\right]\\
&\leq B-A.
\end{align*}		
\end{itemize}
\end{corollary}
\begin{proof}
Let $\alpha=0$ and $\beta=1.$ 
By Proposition \ref{3inv} and Macdonald Theorem, 
\begin{align*}
{\rm I}&=2\left[A-2\{A(A+B)^{-1}A\}\right]\\
{\rm I'}&=(\ln\delta)A+2\left[A-2\delta \{A(\delta A+B)^{-1}A\}\right]\\
{\rm II'}&=(\ln\delta+4)A-8\sqrt{\delta}\left\{A(A\#_{\frac{1}{2}}B+\sqrt{\delta}A)^{-1}A\right\} \\
{\rm III'}&=	\left( \frac{1}{\sqrt{\delta}}A\#_{\frac{1}{2}}B-\sqrt{\delta} A\#_{-\frac{1}{2}}B \right)+(\ln\delta) A\\
{\rm V'}&=\frac{1}{2} \left( \frac{1}{\delta}B-\delta \{AB^{-1}A\} \right)+(\ln\delta) A\\
{\rm V}&=\frac{1}{2}\left(B-\{AB^{-1}A\}\right)
\end{align*}
 
Proof of (i). Combining Theorem \ref{ulbroe}(i), Corollary \ref{ulbroe1}, and Theorem \ref{ulbroe3} (i), (ii), we obtain desired inequalities
\begin{align*}
A-\{AB^{-1}A\}\leq {\rm I}\leq {\rm II}\leq {\rm II'}\leq S_{\alpha, \beta}(A|B)\leq {\rm III'}\leq {\rm V'}\leq {\rm V}\leq B-A.	
\end{align*}

Similar arguments as in the proof of (i) gives (ii).	
\end{proof}


The following result is the ordering relation between Tsallis relative operator $(\lambda, \beta)$-entropy and relative operator $(0, \beta)$-entropy in the setting of JB-algebras. 
\begin{proposition}
Let $A$ and $B$ be invertible positive elements in a unital JB-algebra $\Al.$  For any $0<\lambda \leq 1$ and $\beta>0$ we have 
\begin{align*}
T_{-\lambda, \beta}(A\vert B)\leq S_{0, \beta}(A\vert B)\leq T_{\lambda, \beta}(A\vert B).	
\end{align*}	
\end{proposition}
\begin{proof}
One sees that for any $0<\lambda\leq 1,$
\begin{align}\label{lnllog}
\ln_{-\lambda} x\leq \log x	\leq \ln_{\lambda} x
\end{align}
hold for all $x>0.$ 
Applying Theorem \ref{tnapf} to (\ref{lnllog}) with $h(t)=t^{\beta}$, we derive the inequalities
\begin{align*}
T_{-\lambda, \beta}(A\vert B)\leq S_{0, \beta}(A\vert B)\leq T_{\lambda, \beta}(A\vert B).		
\end{align*}
	
\end{proof}\begin{proposition}\label{Pgtroei1}
For any positive invertible elements $A$ and $B$ in a unital JB-algebra $\Al,$ $0<\lambda\leq 1$ and $\beta>0,$
\begin{align}\label{gtroei1}
A\#_{(0,\beta)}B-	A\#_{(-1,\beta)}B\leq T_{\lambda ,\beta}(A\vert B)\leq A\#_{(1,\beta)}B-	A\#_{(0,\beta)}B.
\end{align}
Moreover, $T_{\lambda ,\beta}(A\vert B)=0$ if and only if $A^{\beta}=B.$
\end{proposition}
\begin{proof}
Note that for any $0< \lambda\leq 1$ and $x>0$
\begin{align}\label{lnllog1}
1-x^{-1}\leq \ln_{\lambda}x \leq x-1.
\end{align}
Using perspective function associated with the three functions above with $h(t)=t^{\beta}$ and applying Theorem \ref{tnapf} to (\ref{lnllog1}), we obtain desired result.

By MacDonald's theorem,
\begin{align*}
A\#_{(0,\beta)}B-	A\#_{(-1,\beta)}B&=A^{\beta}-\{A^{\beta}B^{-1}A^{\beta}\}\\
A\#_{(1,\beta)}B-	A\#_{(0,\beta)}B&=B-A^{\beta}	
\end{align*}
Suppose that $T_{\lambda, \beta}(A\vert B)=0,$ then
\begin{align*}
A^{\beta}-\{A^{\beta}B^{-1}A^{\beta}\}\leq 0\leq B-A^{\beta}.
\end{align*}
Consequently,
\begin{align*}
A^{-\beta}-B^{-1}\leq 0\leq B-A^{\beta}.	
\end{align*}
According to \cite[Lemma 3.5.3]{Hanche84jordan}, $A^{\beta}\leq B\leq A^{\beta}.$
\end{proof}
\begin{remark}
If $\beta=1,$ then $(\ref{gtroei1})$ becomes
\begin{align}\label{troei}
A-\{AB^{-1}A\}\leq T_{\lambda}(A\vert B)\leq -A+B.
\end{align}
If in addition $\Al$ is special, then {\rm Proposition \ref{Pgtroei1}} reduces to  \cite[Proposition 3.4]{Furuichi05Tsallis}.\end{remark}
Denote 
$${\rm IV}=\frac{1}{2}\left(A\#_{(\lambda, \beta)}B-A\#_{(\lambda-1, \beta)}B+A\#_{(1, \beta)}B-A\#_{(0, \beta)}B\right).$$	

The following result provides an improvement for the lower and upper bounds of Tsallis relative operator $(\lambda, \beta)$-entropy.
\begin{proposition}
Let $A$ and $B$ be two positive invertible elements in a unital JB-algebra $\Al,$ $0<\lambda\leq 1$ and $\beta>0.$
\begin{itemize}
\item[(i)] If $A^{\beta}\leq B,$ then
\begin{align}
A\#_{(0, \beta)}B-A\#_{(-1, \beta)}B
&\leq A\#_{(\lambda, \beta)}B-A\#_{(\lambda-1, \beta)}B \nonumber\\
&\leq T_{\lambda ,\beta}(A\vert B)\nonumber\\
&\leq {\rm IV}
\leq A\#_{(1,\beta)}B-	A\#_{(0,\beta)}B\nonumber	
\end{align}
\item[(ii)] If $B\leq A^{\beta},$ then
\begin{align}
A\#_{(0, \beta)}B-A\#_{(-1, \beta)}B
&\leq A\#_{(\lambda, \beta)}B-A\#_{(\lambda-1, \beta)}B\nonumber\\
&\leq {\rm IV}\leq T_{\lambda ,\beta}(A\vert B)\nonumber\\
&\leq A\#_{(1,\beta)}B-	A\#_{(0,\beta)}B \nonumber	
\end{align}	
\end{itemize}
\end{proposition}
\begin{proof}
Proof of (i).
For any $x\geq 1$ and $0<\lambda\leq 1,$ we have the following inequalities
\begin{align}\label{rftroei2}
1-\frac{1}{x}
\leq x^{\lambda}-x^{\lambda-1}
\leq \ln_{\lambda}x
\leq \frac{1}{2}(x^{\lambda}-x^{\lambda -1}+x-1)
\leq x-1.	
\end{align}
Applying Theorem \ref{tnapf} to (\ref{rftroei2}) with $h(t)=t^{\beta},$ we obtain desired result. 

For (ii), applying Hermite-Hadamard integral inequality to  $f(t)=t^{\lambda-1}$ on $[x, 1],$   
\begin{align}\label{rftroei3}
1-\frac{1}{x}
\leq x^{\lambda}-x^{\lambda-1}
\leq \frac{1}{2}(x^{\lambda}-x^{\lambda -1}+x-1)
\leq \ln_{\lambda}x
\leq x-1
\end{align}	
Using the perspective functions associated with the functions in (\ref{rftroei3}) and applying Theorem \ref{tnapf}, (ii) follows.
\end{proof}


\end{document}